\documentclass[10pt]{article}

\usepackage{amsmath}
\usepackage{amsfonts}
\usepackage{amsthm}
\usepackage[english]{babel}
\usepackage{graphicx}
\usepackage[all]{xy}
\usepackage{array,calc,youngtab}
\usepackage{amssymb,stmaryrd}
\usepackage{amscd}

\newtheorem{theorem}{Theorem}
\newtheorem{lemma}{Lemma}
\newtheorem{definition}{Definition}
\newtheorem{corollary}{Corollary}
\newtheorem{proposition}{Proposition}
\newtheorem{remark}{Remark}

\newcommand{\mC}{\mathbb{C}}

\newcommand{\mZ}{\mathbb{Z}}

\newcommand{\mF}{\mathbb{F}}
\newcommand{\mV}{\mathbb{V}}
\newcommand{\mW}{\mathbb{W}}

\newcommand{\mg}{\mathfrak{g}}
\newcommand{\ma}{\mathfrak{a}}
\newcommand{\Mp}{\mathfrak{p}}

\newcommand{\mn}{\mathfrak{n}}
\newcommand{\mb}{\mathfrak{b}}

\newcommand{\cC}{\mathcal{C}}

\newcommand{\cO}{\mathcal{O}}

\newcommand{\cU}{\mathcal{U}}

\newcommand{\cJ}{\mathcal{J}}

\begin{document}
\title{Bernstein-Gelfand-Gelfand resolutions for linear superalgebras}

\author{K.\ Coulembier\thanks{Postdoctoral Fellow of the Research Foundation - Flanders (FWO), E-mail: {\tt Coulembier@cage.ugent.be}} }

\date{\small{Department of Mathematical Analysis\\
Faculty of Engineering and Architecture -- Ghent University\\ Krijgslaan 281, 9000 Gent,
Belgium}\\
}

\maketitle

\begin{abstract}
In this paper we construct resolutions of finite dimensional irreducible $\mathfrak{gl}(m|n)$-modules in terms of generalized Verma modules. The resolutions are determined by the Kostant cohomology groups and extend the strong (Lepowsky-)Bernstein-Gelfand-Gelfand resolutions to the setting of Lie superalgebras. It is known that such resolutions for finite dimensional representations of Lie superalgebras do not exist in general. Thus far they have only been discovered for $\mathfrak{gl}(m|n)$ in case the parabolic subalgebra has reductive part equal to $\mathfrak{gl}(m)\oplus\mathfrak{gl}(n)$ and for tensor modules. In the current paper we prove the existence of the resolutions for tensor modules of $\mathfrak{gl}(m|n)$ or $\mathfrak{sl}(m|n)$ and their duals for an extensive class of parabolic subalgebras including the ones already considered in the literature. 
\end{abstract}


\section{Introduction and overview of main results}
\label{intro}

The BGG resolutions are resolutions of representations of Lie algebras in terms of (generalized) Verma modules, see \cite{MR0578996, lepowsky}. In \cite{MR0578996} this was applied to obtain a simple proof of Bott's theorem. These resolutions and their corresponding morphisms between Verma modules have an interesting dual side in terms of invariant differential operators, see \cite{MR1856258, Cap, MR2861216}. These differential operators have applications in many areas, see references in \cite{MR1856258, Cap, MR2861216} or \cite{MR2180410} and \cite{MR2372762} for concrete applications. Such resolutions have also already been constructed for Kac-Moody algebras in \cite{MR066169} and for some infinite dimensional Lie superalgebras in \cite{MR2646304}. They are also known to exist for unitary infinite dimensional representations for orthosymplectic superalgebras, see \cite{ChLW}. In Section \ref{classical} we give a brief historical overview of the development of BGG resolutions and the corresponding differential operators, which is relevant to explain the approach taken in the current paper. 

In \cite{Cheng} some interesting results on BGG-resolutions for Lie superalgebras were obtained. First of all, the classical results can not be obtained in full generality since the natural module for $\mathfrak{gl}(1|2)$ or $\mathfrak{sl}(1|2)$ does not have a resolution in terms of Verma modules (with parabolic subalgebra equal to the Borel subalgebra). However, it was proven that every tensor module of $\mathfrak{gl}(m|n)$ (covariant module, see \cite{MR0884183}) has a resolution in terms of Kac modules. The Kac modules are (generalized) Verma modules for which the parabolic subalgebra has reductive subalgebra equal to the underlying Lie algebra $\mathfrak{gl}(m)\oplus\mathfrak{gl}(n)$. An essential method used in \cite{Cheng} is the so-called super duality. This gives a correspondence between properties of the Lie superalgebra $\mathfrak{gl}(m|n)$ and the Lie algebra $\mathfrak{gl}(m+n)$ in the limit $n\to\infty$. This method has already led to great success in determining e.g. character formulas and cohomology groups, see \cite{MR1937204, MR2646304, Cheng, ChLW, MR2037723, MR2036954}.

In the current paper we considerably extend the result of \cite{Cheng} to the statement that every tensor module of $\mathfrak{gl}(m|n)$ has a resolution in terms of Verma modules for every parabolic subalgebra containing $\mathfrak{gl}(n)$, naturally embedded in $\mathfrak{gl}(m|n)$. The case in \cite{Cheng} is included in this statement. We also obtain an analogous result for the dual representations of tensor modules (contravariant modules), see Theorem \ref{mainthmgl}. To obtain these results we do not use the machinery of super duality but a reformulation in terms of dual representations of Verma modules (coinduced modules). Then we proceed inspired by the results on the differential operator side in the classical (curved) case in \cite{MR1856258, Cap}. The methods that we use lead to invariant differential operators on super parabolic geometries, but in this paper the material is written in an algebraical way. The extension to differential operators is mainly a matter of definitions and introduction of Lie supergroups and will therefore be included in a forthcoming paper with other results on such invariant differential operators. We hope that this method can have other applications in the study of the parabolic category $\cO$ for Lie superalgebras. In particular it can be expected that the parabolic category $\cO$ corresponding to the parabolic subalgebras we study possesses a rich structure. This is also already obvious from Theorem \ref{Thmquabla} and Corollary \ref{submodVerma}.

In this paper $\mg=\mg_{\overline{0}}\,+\,\mg_{\overline{1}}$ will always stand for a finite dimensional complex semisimple Lie superalgebra and $\Mp$ a parabolic subalgebra (a subalgebra containing a Borel subalgebra of $\mg$). The reductive subalgebra of $\Mp$ is denoted by $\mg_0$, not to be confused with $\mg_{\overline{0}}$, the underlying Lie algebra of $\mg$. The parabolic subalgebra then has a vector space decomposition $\Mp=\mg_0+\mn$ and $\mg$ decomposes likewise as $\mg=\overline{\mn}+\mg_0+\mn$. We also use the notation $\Mp^\ast=\overline{\mn}+\mg_0$. The notations $\mV,\mW$ will be used for finite dimensional irreducible $\mg$-modules. If we want to mention explicitly the highest weight $\lambda$ of the representation it is denoted by $\mV_\lambda$. The notation $[\cdot,\cdot]_\ma$ will always stand for the projection of the Lie superbracket onto a subspace $\ma\subset\mg$ with respect to a naturally defined complement space $\mb$, $\mg=\ma+\mb$. 

The main result of this paper is that there exists a resolution of the $\mg$-module $\mV^\ast$ (dual representation of $\mV$) in terms of parabolic Verma modules with respect to $\Mp$ if the triple $(\mg,\Mp,\mV)$ satisfies the conditions:
\begin{equation}
\label{neccond}
\begin{cases}\circ&\mV\mbox{ is a star representation of }\mg \mbox{ with star map } \dagger \\
\circ&\mbox{There is a weight basis }\{X_a\}\mbox{ of }\mn\mbox{ such that the bilinear form } (\cdot,\cdot): \overline{\mn}\times\mn\to\mC\\
&\mbox{given by }(X_a^\dagger,X_b)=\delta_{ab}\mbox{ satisfies }\left([Y,A]_{\overline{\mn}},X\right)=\left(Y,[A,X]_{\mn}\right)\mbox{ for all }X\in\mn,\,Y\in\overline{\mn}\mbox{ and }A\in\mg.\end{cases}
\end{equation}

For completeness we recall the notation of a star representation for a Lie (super)algebra, see \cite{MR0424886}. A star Lie superalgebra is equipped with a map $\dagger$ which is antilinear, even and satisfies $[A,B]^\dagger=[B^\dagger,A^\dagger]$ and $\left(A^\dagger\right)^\dagger=A$ for $A,B\in\mg$. A star representation of such an algebra is a representation with a (positive definite, hermitian) inner product which satisfies $(Av,w)=(v,A^\dagger w)$ for $v,w\in\mV$ and $A\in\mg$. The only simple finite dimensional Lie superalgebras (excluding Lie algebras) which have finite dimensional star representations are $\mathfrak{sl}(m|n)$, $\mathfrak{osp}(2|2n)$ and $Q(n)$, see \cite{MR0424886}. The properties \eqref{neccond} are quite restrictive but still have interesting classes of examples, as will be shown in Section \ref{ExamGL}. In particular we obtain classes which include the case in \cite{Cheng} but are much larger. In case $\mg$ is a Lie algebra, properties \eqref{neccond} always hold. This follows from the compact real form, which has the invariant negative definite Killing form and for which representations are unitary. 

The combination of the results in the subsequent Theorem \ref{tensor} and Theorem \ref{finalthm} lead to the main theorem of this paper. The notion of tensor modules will be explained in Section \ref{ExamGL}, see also \cite{MR0884183}, and we consider the natural block-diagonal embedding $\mathfrak{g}_{\overline{0}}=\mathfrak{gl}(m)\oplus\mathfrak{gl}(n)\hookrightarrow\mathfrak{gl}(m|n)$.
\begin{theorem}
\label{mainthmgl}
For $\mg=\mathfrak{gl}(m|n)$ with a parabolic subalgebra $\Mp$, the $\mathfrak{g}$-module $\mV$ can be resolved in terms of direct sums of $\Mp$-Verma modules if
\begin{itemize}
\item $\mV$ is a tensor module and $\Mp$ contains $\mathfrak{gl}(n)$
\item $\mV$ is the dual of a tensor module and $\Mp$ contains $\mathfrak{gl}(m)$.
\end{itemize}
\end{theorem}
The explicit form of the resolution can be found in Theorem \ref{finalthm}.

Theorem \ref{mainthmgl} clearly includes the case where $\mV$ is a tensor module and $\mg_0=\mathfrak{gl}(m)\oplus\mathfrak{gl}(n)$, for which the BGG resolutions were constructed in \cite{Cheng}, although we use completely different methods. It also includes the interesting case $\mathfrak{gl}(p+q|n)$ with $\mg_0=\mathfrak{gl}(p)\oplus\mathfrak{gl}(q|n)$ for which the Kostant cohomology was investigated for infinite dimensional representations in \cite{MR2037723}. However, Theorem \ref{mainthmgl} also includes many cases where $\mg_{\overline{0}}$ is not contained in $\mg_0$. In particular for $\mg=\mathfrak{gl}(m|1)$ and $\mV$ a tensor module, there is no restriction on the parabolic subalgebra.

Note that this theorem does not contradict the counterexample of Verma module resolutions given in Example 5.1 in \cite{Cheng}, where it is proven that the natural module $\mC^{1|2}$ of $\mathfrak{gl}(1|2)$ does not have such a resolution for the parabolic subalgebra equal to the Borel subalgebra. Theorem \ref{mainthmgl} does state that the $\mathfrak{gl}(m|1)$-module $\mC^{m|1}$ and the $\mathfrak{gl}(1|n)$-module $\left(\mC^{1|n}\right)^\ast$ have resolutions for $\Mp$ equal to the Borel subalgebra. This shows why the counterexample in \cite{Cheng} is the easiest possible counterexample and thus that already the easiest case not included in Theorem \ref{mainthmgl} can not be used to extend the result.

The remainder of the paper is organized as follows. In Section \ref{classical} we give an overview of the classical results on BGG resolutions. In Section \ref{ExamGL} we prove that $\mathfrak{gl}(m|n)$ leads to a class of cases on which our general approach is applicable. In Section \ref{Kostant} we recall some notions on Kostant cohomology for Lie superalgebras in a setting that will be useful for the sequel and we derive some explicit formulae for the relevant operators. In particular we obtain a useful harmonic theory with a Kostant quabla operator. We prove that as in the classical case this operator can be expressed in terms of Casimir operators. In Section \ref{BGG} we derive the BGG-type resolutions by constructing the dual picture on coinduced modules. We construct a coresolution corresponding to a local twisted de Rham sequence. Using the obtained properties of the Kostant cohomology we can then derive a resolution in terms of Verma modules. We also discuss some immediate consequences of the resolutions. Finally we include an Appendix where we obtain relevant information to calculate the cohomology of the twisted de Rham sequence, using Hopf algebraical techniques.

\section{Lepowsky BGG resolutions and invariant differential operators}
\label{classical}
In \cite{MR0578996}, Bernstein, Gel'fand and Gel'fand proved that each finite dimensional irreducible representation of a complex semisimple Lie algebra has a resolution in terms of Verma modules. In \cite{lepowsky} Lepowsky proved that this extends to the parabolic case. To state this result we need to introduce the notation $M(\lambda)$ for the irreducible $\Mp$-representation with highest weight $\lambda\in\mathfrak{h}^\ast$ with $\mathfrak{h}$ the Cartan subalgebra of $\mg$. For the $\mg$-representation $\mV_\lambda$ there is an exact sequence 
\begin{eqnarray*}
0\rightarrow\bigoplus_{w\in W^1(\dim\mn)}V^{M(w\cdot\lambda)} \to\cdots\to \bigoplus_{w\in W^1(j)}V^{M(w\cdot\lambda)}\to\cdots\to \bigoplus_{w\in W^1(1)}V^{M(w\cdot\lambda)}\to V^{M(\lambda)} \to \mV_\lambda\to0
\end{eqnarray*}
with $W^1$ the parabolic Weyl group (with $\rho$-shifted action) and $V^M$ the parabolic Verma module generated by $M$, see \cite{lepowsky} for details.

In \cite{MR0142696}, Kostant investigated a certain (co)homology, which will be explained in Section \ref{Kostant}. The main result there was that the relevant homology groups satisfy
\begin{eqnarray*}
H^j(\mn,\mV)&\cong& \bigoplus_{w\in W^1(j)}M(w\cdot\lambda)
\end{eqnarray*}
as $\Mp$-modules. The result of Lepowsky can therefore be rewritten in terms of these cohomology groups.

There is a well-known correspondence between Verma module morphisms and differential operators between the principal series representations, corresponding to vector bundles on the generalized flag manifolds $G/P$ with $G$ and $P$ groups with Lie algebras $\mg$ and $\Mp$. This implies that there is a locally exact sequence
\begin{eqnarray*}
&&0\quad\to\quad \mV^\ast\quad\to\quad\Gamma(G/P,G\times_PH^0(\overline{\mn},\mV^\ast))\quad\to\quad\Gamma(G/P,G\times_PH^1(\overline{\mn},\mV^\ast))\quad\to\,\cdots\\
&&\cdots\,\to\quad\Gamma(G/P,G\times_PH^{\dim \mn}(\overline{\mn},\mV^\ast))\quad\to \quad0.
\end{eqnarray*}
One of the interesting features of this result is that each irreducible representation $\mV^\ast$ can be explicitly realized as the kernel of some set of differential operators.

In \cite{Cap} \v{C}ap, Slov\'ak and Sou{\v{c}}ek proved that the differential operators in the sequence above extend to curved Cartan-geometries bases on $G/P$, known as parabolic geometries. In doing so they also provided a new proof of the BGG resolutions for the flat model $G/P$. It turns out that the approach using differential operators extends more easily to the supersetting than the direct proof of the BGG resolutions. In this paper we use methods inspired by the simplification of \cite{Cap} provided in \cite{MR1856258} to prove BGG resolutions for the Lie superalgebra $\mathfrak{gl}(m|n)$. Even though some of the machinery is inspired by the classical differential operator side we will formulate and prove everything in a purely algebraic setting in the current paper.

\section{General linear superalgebra and tensor product representations}
\label{ExamGL}

In this section we prove that we have an extensive class of representations and parabolic subalgebras of $\mathfrak{gl}(m|n)$ that satisfy the conditions \eqref{neccond}. Working with $\mathfrak{gl}(m|n)$ or its simple subalgebra $\mathfrak{sl}(m|n)$ makes no difference, so Theorem \ref{tensor} also holds for $\mathfrak{sl}(m|n)$. The representations correspond to tensor modules or their duals. So even though the conditions are restrictive, the preserved class of cases is worth studying and, as will become apparent, includes all cases for which the BGG resolutions (for finite dimensional representations) are already known to exist, see \cite{Cheng}. The dual of a representation $\mV$ is defined as the space of linear functionals $\mV^\ast$ with action of $\mg$ given by $(A\alpha)(v)=-(-1)^{|A||\alpha|}\alpha(Av)$ for $A\in\mg$, $\alpha\in\mV^\ast$ and $v\in\mV$.

The natural module $\mC^{m|n}$ of $\mathfrak{gl}(m|n)$ or $\mathfrak{sl}(m|n)$ is a star representation. Therefore all its tensor powers are completely reducible. All modules appearing as submodules of these tensor powers are called tensor modules. The highest weights of these modules are described in \cite{Cheng, MR2037723, MR2036954}. They can be obtained by the superduality between $\mathfrak{gl}(m|\infty)$ and $\mathfrak{gl}(m+\infty)$. They are also easily described in terms of Hook Young diagrams, see \cite{MR0884183}.

Since the exterior powers do not have a top form for $\mC^{m|n}$ the module $\left(\mC^{m|n}\right)^\ast$ does not appear as a tensor module contrary to the classical case. This module also corresponds to a star-representation but this star map is different than the one from $\mC^{m|n}$, see the subsequent Proposition \ref{stardual}. So all duals of tensor modules appear as submodules of the tensor powers of $\left(\mC^{m|n}\right)^\ast$.

\begin{theorem}
\label{tensor}
The triple $(\mathfrak{gl}(m|n),\Mp,\mV)$ for $\Mp$ a parabolic subalgebra of $\mathfrak{gl}(m|n)$ satisfies conditions \eqref{neccond} if
\begin{itemize}
\item $\mV$ is a tensor module and $\Mp$ contains $\mathfrak{gl}(m)$
\item $\mV^\ast$ is a tensor module and $\Mp$ contains $\mathfrak{gl}(n)$.
\end{itemize}
\end{theorem}
\begin{proof}
The theorem is proven if it holds for $\mV=\mC^{m|n}$ and $\mV=\left(\mC^{m|n}\right)^\ast$. First we consider $\mV=\mC^{m|n}$. The easiest realization of $\mC^{m|n}$ is as $m$ commuting variables and $n$ anti-commuting ones. These are denoted by $\{x_j|j=1,\cdots,m+n\}$ and satisfy commutation relations
\begin{eqnarray*}
x_ix_j=(-1)^{[i][j]}x_jx_i&\mbox{with}&[k]=\overline{0} \mbox{ (respectively $\overline{1}$) if $k\le m$ (respectively $k>m$)}.
\end{eqnarray*}
The corresponding partial derivatives are then defined by the Leibniz rule $\partial_{x_i}x_j=\delta_{ij}+(-1)^{[i][j]}x_j\partial_{x_i}$. The differential operators $\{x_i\partial_{x_j}|i,j=1,\cdots,m+n\}$ generate a Lie superalgebra isomorphic to $\mathfrak{gl}(m|n)$. The inner product on $\mC^{m|n}$ is then given by $\langle x_i,x_j\rangle=\delta_{ij}$, which leads to star map
\begin{eqnarray*}
\left(x_i\partial_{x_j}\right)^\dagger=x_j\partial_{x_i}.
\end{eqnarray*}
The Borel subalgebra corresponds to $\{x_i\partial_{x_j}|i\le j\}$. Saying that $\Mp$ contains $\mathfrak{gl}(m)$ implies that $\mn$ is contained in the subalgebra $N=\{x_i\partial_{x_j}|i< j,m<j\}\subset\mathfrak{gl}(m|n)$. We also introduce the notation $\overline{N}=\{x_i\partial_{x_j}|j<i,m<i\}$ for the subalgebra which contains $\overline{\mn}$. It remains to be checked that the bilinear form $\overline{N}\times N\to\mC$ defined by
\begin{eqnarray}
\label{defbilform}
\left(x_k\partial_l,x_i\partial_{x_j}\right)&=&\delta_{il}\delta_{jk}
\end{eqnarray}
satisfies the property
\begin{eqnarray*}
\left([x_k\partial_l,x_a\partial_{x_b}]_{\overline{\mn}},x_i\partial_{x_j}\right)&=&\left(x_k\partial_l,[x_a\partial_{x_b},x_i\partial_{x_j}]_{{\mn}}\right)
\end{eqnarray*}
for $x_i\partial_{x_j}\in\mn$ and $x_k\partial_l\in\overline{\mn}$ for each parabolic such that $\mn\subset N$. Since the bilinear form \eqref{defbilform} is zero except when the weights of left and right hand side add up to zero, this statement is proved by proving that
\begin{eqnarray*}
\left([x_k\partial_l,x_a\partial_{x_b}],x_i\partial_{x_j}\right)&=&\left(x_k\partial_l,[x_a\partial_{x_b},x_i\partial_{x_j}]\right)
\end{eqnarray*}
holds for $x_a\partial_{x_b}\in\mathfrak{gl}(m|n)$, $x_i\partial_{x_j}\in N$ and $x_k\partial_l\in\overline{N}$. This property then follows from a trivial calculation taking into account that $[j]=[k]=1$.

The reasoning for $V=\left(\mC^{m|n}\right)^\ast$ is very similar. This representation can be realized by the partial derivatives $\{\partial_{x_i}\}$ with $\mathfrak{gl}(m|n)$-action given by the super commutator. The inner product is defined by $\langle\partial_{x_i},\partial_{x_j}\rangle=\delta_{ij}$. The star map on $\mathfrak{gl}(m|n)$ is then given by $\left(x_i\partial_{x_j}\right)^\dagger=(-1)^{[i]+[j]}x_j\partial_{x_i}$.
\end{proof}
In particular if we take $\mV$ to be a tensor module and $\Mp$ the parabolic subalgebra with reductive subalgebra $\mg_{\overline{0}}=\mathfrak{gl}(m)\oplus\mathfrak{gl}(n)$, so $\mn\cong\mC^{0|mn}$, the necessary requirements are met. This corresponds to the case studied in \cite{Cheng}.

It is a general result that the dual of a star representation is also a star representation, but with different star map as is stated in the following proposition.
\begin{proposition}
\label{stardual}
If the irreducible $\mg$-module $\mV$ is a star representation for star map $\dagger$, then $\mV^\ast$ is also a star representation for star map
\begin{eqnarray*}
A&\to& (-1)^{|A|}A^\dagger\qquad\mbox{ for }A\in\mg.
\end{eqnarray*}
\end{proposition}
\begin{proof}
The inner product on $\mV^\ast$ is defined as $\langle \alpha,\beta\rangle=\sum_i\alpha(v_i)\overline{\beta(v_i)}$ for $\{v_i\}$ an orthonormal basis of $\mV$. The fact that $\langle A\alpha,\beta\rangle=(-1)^{|A|}\langle \alpha,A^\dagger \beta\rangle$ follows from a direct calculation. Also the fact that the proposed mapping satisfies the right definitions of a star map (given in Section \ref{intro}) follows in a straightforward manner.
\end{proof}

This result allows to formulate conditions \eqref{neccond} directly in terms of $\mV^\ast=\mW$. The conditions on the triple $(\mg,\Mp,\mW)$ then become
\begin{equation}
\label{neccond2}
\begin{cases}\circ&\mW\mbox{ is a star representation of }\mg \mbox{ with star map } \dagger \\
\circ&\mbox{There is a weight basis }\{X_a\}\mbox{ of }\mn\mbox{ such that the bilinear form } (\cdot,\cdot): \overline{\mn}\times\mn\to\mC\mbox{ given by}\\
&(X_a^\dagger,X_b)=(-1)^{|X_a|}\delta_{ab}\mbox{ satisfies }\left([Y,A]_{\overline{\mn}},X\right)=\left(Y,[A,X]_{\mn}\right)\mbox{ for all }X\in\mn,\,Y\in\overline{\mn}\mbox{ and }A\in\mg.\end{cases}
\end{equation}

We end this section with a remark to explain why the result on the parabolic category $\cO$ are stronger for the cases that satisfy conditions \eqref{neccond}.
\begin{remark}
\label{Killing}
{\rm The normalized Killing form on $\mathfrak{gl}(m|n)$ is given by $(x_k\partial_{x_l},x_i\partial_{x_j})=(-1)^{[k]}\delta_{kj}\delta_{li}$ for arbitrary $x_k\partial_{x_l}$ and $x_i\partial_{x_j}$. For Lie algebras the star map of an irreducible representation will always satisfy the property that it maps a basis of $\mg$ to its dual with respect to this form. For Lie superalgebras this is never the case. For instance for $\mathfrak{gl}(m|n)$ this follows from the fact that the adjoint representation is $\mathfrak{gl}(m|n)=\mC^{m|n}\otimes\left(\mC^{m|n}\right)^\ast$, a mixture of a covariant and contravariant representation.  However if we only look at the vectors in $\overline{\mn}+\mn\subset\mg$ and restrict the Killing form, this dual basis property can still hold. This is exactly what is done in Theorem \ref{tensor}. By taking $\mg_0$ large enough one reobtains some results on Kostant cohomology and BGG resolutions more closely related to the classical case.}
\end{remark}

\section{Kostant cohomology}
\label{Kostant}

The type of (co)homology we need was originally studied by Kostant for Lie algebras in \cite{MR0142696}. For Lie superalgebras results have been obtained in e.g. \cite{MR1937204, MR2646304, Cheng, MR2037723, MR2036954}.

From now on we assume that we have a fixed triple $(\mg,\Mp,\mV)$, in the notation of the introduction, of a semisimple Lie superalgebra $\mg$ with parabolic subalgebra and finite dimensional irreducible representation $\mV$ satisfying the properties \eqref{neccond} although many results could also be formulated without these conditions. As in the introduction we use the notation $\mg=\overline{\mn}+\mg_0+\mn$. We also use the notation $|\cdot|\in\mZ_2$ which maps a homogeneous element of a super vector space to $\overline{0}$ or $\overline{1}$ depending on whether it is even or odd. The summation $\sum_a$ will always be used for a summation $\sum_{a=1}^{\dim\mn}$ related to the basis $\{X_a\}$ of $\mn$ corresponding to properties \eqref{neccond}.

Kostant (co)homology for the Lie superalgebra $\mg$ with parabolic subalgebra $\Mp$ is defined by introducing the space of $k$-chains $C^k(\overline{\mn},\mV)=\Lambda^k(\mn)\otimes\mV$. These spaces are naturally $\Mp$-modules. The inner product $\langle\cdot,\cdot\rangle$ on $\mn$ is defined by $\langle X_1,X_2\rangle=(X_1^\dagger,X_2)$ with $(\cdot,\cdot)$ the bilinear form in properties \eqref{neccond}. This inner product extends to the tensor powers $\otimes^k\mn$ and therefore also to the superantisymmetric powers $\Lambda^k(\mn)$. The inner product on $C^k(\overline{\mn},\mV)$ is then given by the relation 
\begin{eqnarray}
\label{prodinnerprod}
\langle \alpha\otimes v,\beta\otimes w\rangle&=&\langle \alpha,\beta\rangle\langle v,w\rangle
\end{eqnarray}
for $\alpha,\beta\in\Lambda^k(\mn)$ and $v,w\in\mV$. Here we have also used the notation $\langle\cdot,\cdot \rangle$ for the inner product on $\mV$ which is assumed to exist by properties \eqref{neccond}. The definition of the bilinear form $(\cdot,\cdot)$ in properties \eqref{neccond} implies that $\langle\cdot,\cdot\rangle$ is hermitian and positive definite. The notation of the wedge product $X^1\wedge X^2$ stands for $\frac{1}{2}(X^1\otimes X^2-(-1)^{|X^1||X^2|}X^2\otimes X^1)$ and likewise for higher orders.

For $A\in\mg$ and $f\in C^k(\overline{\mn},\mV)$, $A[f]$ is defined as
\begin{eqnarray}
\label{gactionchain}
A[X\wedge g]=[A,X]_{\mn}\wedge g+X\wedge A[g]&\mbox{and}&A[v]=A\cdot v\mbox{ for }v\in\mV.
\end{eqnarray}
Note that for $A\in\Mp$, $A[f]=A\cdot f$ holds. If we restrict $A$ to be an element of $\Mp^\ast$, this introduces a $\Mp^\ast$-representation structure on $C^k(\overline{\mn},\mV)$, which corresponds to the quotients $\left(\Lambda^k\mg\otimes\mV\right)/\left(\Lambda^k\Mp^\ast\otimes\mV\right)$ which are isomorphic as vector spaces to $\Lambda^k\mn\otimes\mV$.

\begin{lemma}
\label{conscond}
If conditions \eqref{neccond} are satisfied, the modules $C^k(\overline{\mn},\mV)$ are star representations for $\mg_0$ with inner product introduced above and star map on $\mg_0$ the one inherited from $\mg$. Furthermore they satisfy the property
\begin{eqnarray*}
\langle Z\cdot f,g\rangle&=&\langle f, Z^\dagger[g]\rangle\quad\mbox{ for all }Z\in\Mp\mbox{ and }f,g\in C^k(\overline{\mn},\mV).
\end{eqnarray*}
\end{lemma}

In our approach it is most natural to start from the codifferential 
\begin{equation}
\label{defcodiff}
\partial^\ast_k:C^k(\overline{\mn},\mV)\to C^{k-1}(\overline{\mn},\mV)\qquad \partial^\ast_k(X\wedge f)=-X\cdot f- X\wedge \partial_{k-1}^\ast (f),
\end{equation}
where we set $\partial_0^\ast=0$. It can be checked by straightforward calculations and induction that $\partial^\ast_k\circ\partial^\ast_{k+1}=0$ and $\partial^\ast_k$ is a $\Mp$-module morphism.

The cohomology groups are the $\Mp$-modules defined as
\begin{eqnarray}
H^k(\overline{\mn},\mV)&=& \ker(\partial_k^\ast)/\mbox{im}(\partial_{k+1}^\ast). 
\end{eqnarray}

To develop a Hodge theory we introduce the adjoint of $\partial^\ast$, which corresponds to the standard differential, $\langle \partial f,g\rangle=-\langle f,\partial^\ast g\rangle$, for $\langle\cdot,\cdot\rangle$ the inner product on $C^k(\overline{\mn},\mV)$ defined above.

By its definition and Lemma \ref{conscond}, $\partial_k:C^k(\overline{\mn},\mV)\to C^{k+1}(\overline{\mn},\mV)$ is a $\mg_0$-module morphism and satisfies $\partial_k\circ\partial_{k-1}=0$. It also follows straight away that $\partial$ and $\partial^\ast$ are disjoint operators, i.e. $\partial(\partial^\ast f)=0$ implies $\partial^\ast f=0$ and the same for $\partial$ and $\partial^\ast$ reversed. The Kostant quabla operator is then defined as $\Box=\partial\circ\partial^\ast+\partial^\ast\circ\partial$.

The following lemma follows immediately from the general theory in Proposition 2.1 in \cite{MR0142696}.
\begin{lemma}
\label{lemmaKostant}
The following decomposition of $\mg_0$-modules holds: {\rm
\begin{eqnarray}
C^k(\overline{\mn},\mV)&=&\mbox{im}\partial\,\oplus\, \ker \Box\,\oplus\,\mbox{im}\partial^\ast.
\end{eqnarray}
}Moreover, {\rm$\ker \partial=\mbox{im}\partial\,\oplus\, \ker \Box$, $\,\ker\partial^\ast=\ker \Box\,\oplus\,\mbox{im}\partial^\ast$} and {\rm $\,\mbox{im}\Box=\mbox{im}\partial\oplus\mbox{im}\partial^\ast$}.
\end{lemma}
This implies that the following $\mg_0$-module isomorphisms exist:
\begin{eqnarray*}
H^k(\overline{\mn},\mV)\,\cong\, \ker\Box\,\cong\, \ker( \partial_k)/\mbox{im}(\partial_{k-1}).
\end{eqnarray*}
In particular it follows that as a $\mg_0$-module, $H^k(\overline{\mn},\mV)$ is embedded in the $\mg_0$-star representation $C^k(\overline{\mn},\mV)$ and therefore $H^k(\overline{\mn},\mV)$ is completely reducible as a $\mg_0$-module.
\begin{corollary}
\label{Hcompred}
The space $H^k(\overline{\mn},\mV)$ is completely reducible as a $\Mp$-module.
\end{corollary}
\begin{proof}
As argued above, $H^k(\overline{\mn},\mV)$ is completely reducible as a $\mg_0$-module. Equation \eqref{defcodiff} implies that if $f\in\ker\partial^\ast$, then $X\cdot f\in \mbox{im}\partial^\ast$ for $X\in \mn$ and therefore the $\mn$-action is trivial on $H^k(\overline{\mn},\mV)$.
\end{proof}

From condition \eqref{neccond} it follows immediately that for any $A\in\mg$,
\begin{eqnarray*}
[X_a^\dagger,A]_{\overline{\mn}}=\sum_{a}C_{ab}X_b^\dagger\mbox{ for constants }C_{ab}&\mbox{implies}&[A,X_c]_{\mn}=\sum_aC_{ac}X_a.
\end{eqnarray*}
This leads to the useful identity
\begin{eqnarray}
\label{propdualbasis}
\sum_af(|X_a|)\,\, X_a\otimes [X^\dagger_a,A]_{\overline{\mn}}&=&\sum_a f(|X_a|+|A|)[A,X_a]_{\mn}\otimes X_a^\dagger
\end{eqnarray}
for $A\in\mg$ and any function $f:\mZ_2\to\mC$.

The following technical facts about the codifferential will be essential for the sequel. 

\begin{lemma}
\label{propcodiff}
The explicit identity of the codifferential on $C^k(\overline{\mn},\mV)$ is given by
\begin{eqnarray*}
\partial^\ast(X^1\wedge\cdots\wedge X^k\otimes v)=\sum_{i=1}^k(-1)^{i+|X^i|(|X^{i+1}|+\cdots+|X^k|)}X^1\wedge\cdots^{\hat{i}}\cdot\wedge X^k\otimes X^i\cdot v\\
+\sum_{i<j}^k(-1)^{i+j+(|X^i|+|X^j|)(|X^{1}|+\cdots+|X^{i-1}|)+|X^j|(|X^{i+1}|+\cdots+|X^{j-1}|)}[X^i,X^j]\wedge X^1\wedge\cdots^{\hat{i}}\cdots^{\hat{j}}\cdot\wedge X^k\otimes v.
\end{eqnarray*}
For $A\in\mg$ and $f=X^1\wedge\cdots\wedge X^k\otimes v\in C^k(\overline{\mn},\mV)$, the following relation holds:
\begin{eqnarray*}
\partial^\ast(A[f])-A[\partial^\ast f]&=&\sum_{i=1}^k(-1)^{i-1+|X^i|(|X^1|+\cdots+|X^{i-1}|)}\left([A,X^i]_{\Mp^\ast}\right)\left[X^1\wedge\cdots^{\hat{i}}\cdot\wedge X^k\otimes v\right].
\end{eqnarray*}
\end{lemma}
\begin{proof}
The first property follows quickly from the definition of $\partial^\ast$. The second property follows from a straightforward computation using the first property.
\end{proof}
The second statement in this lemma also shows again that $\partial^\ast$ is a $\Mp$-module morphism since $[\Mp,\mn]_{\Mp^\ast}=0$. In order to create BGG sequences we need the following proposition, which describes how far $\partial$ is from being a $\Mp$-module morphism.
\begin{corollary}
\label{pactionder}
For $Z\in\Mp$ and $f\in C^k(\overline{\mn},\mV)$ the following relation holds:
\begin{eqnarray*}
\partial (Z \cdot f)&=&Z\cdot \partial (f)+(k+1)\sum_a X_a\wedge[X_a^\dagger,Z]_{\Mp}\cdot f.
\end{eqnarray*}
\end{corollary}
\begin{proof}
The proof is based on the fact that $\partial$ and $\partial^\ast$ are adjoint operators, Lemma \ref{conscond} and Lemma \ref{propcodiff} and the property $X=\sum_a\langle X_a,X\rangle X_a$ or $X^\dagger=\sum_a\langle X,X_a\rangle X_a^\dagger$ for any $X\in\mn$. For $f\in C^{k-1}(\overline{\mn},\mV)$ and  $g=X^1\wedge\cdots\wedge X^k\otimes v\in C^k(\overline{\mn},\mV)$ we calculate
\begin{eqnarray*}
\langle \partial Z\cdot f,g\rangle -\langle Z\cdot \partial f,g\rangle&=&\sum_{i=1}^k\langle f,\left([Z^\dagger,X^i]_{\Mp^\ast}\right)\left[X^1\wedge\cdots^{\hat{i}}\cdot\wedge X^k\otimes v\right]\rangle(-1)^{i-1+|X^i|(|X^1|+\cdots+|X^{i-1}|)}\\
&=&\sum_a\sum_{i=1}^k\langle X_a,X^i\rangle\langle [X_a^\dagger,Z]_{\Mp}f,X^1\wedge\cdots^{\hat{i}}\cdot\wedge X^k\otimes v\rangle(-1)^{i-1+|X^i|(|X^1|+\cdots+|X^{i-1}|)}\\
&=&k\sum_a\langle X_a\wedge [X_a^\dagger,Z]_{\Mp}\cdot f,X^1\wedge\cdots\wedge X^k\otimes v\rangle
\end{eqnarray*}
which proves the statement.
\end{proof}

\begin{lemma}
\label{propstandder}
The standard derivative satisfies
\begin{eqnarray*}
\partial(X\wedge f)&=&\frac{k+1}{2}\sum_aX_a\wedge [X_a^\dagger,X]_{\mn}\wedge f-\frac{k+1}{k}X\wedge\partial f
\end{eqnarray*}
for $X\in\mn$ and $f\in C^{k-1}(\overline{\mn},\mV)$ and $\partial v=\sum_a X_a\otimes X_a^\dagger\cdot v$ for $v\in\mV=C^0(\overline{\mn},\mV)$.
\end{lemma}
\begin{proof}
The proof of the formula for $C^0(\overline{\mn},\mV)$ follows immediately from the definition of $\partial$ as the adjoint operator of $\partial^\ast$. Starting from Lemma \ref{propcodiff} it can be derived from a tedious calculation that
\begin{eqnarray*}
\partial\left(X^1\wedge\cdots\wedge X^k\otimes v\right)&=&\frac{k+1}{2}\sum_aX_a\wedge [X_a^\dagger,X^1\wedge\cdots\wedge X^k]_{\mn}\otimes v\\
&&+(k+1)(-1)^k\sum_a X^1\wedge\cdots\wedge X^k \wedge X_a\otimes X_a^\dagger\cdot v
\end{eqnarray*}
holds. The result then follows easily.
\end{proof}

We show that the Kostant quabla operator $\Box$ can be expressed in terms of Casimir operators, as in Theorem 4.4 in \cite{MR0142696} for Lie algebras, or Lemma 4.4 in \cite{MR2036954} for $\mg=\mathfrak{gl}(m|n)$ and $\mathfrak{g}_0=\mg_{\overline{0}}=\mathfrak{gl}(m)\oplus\mathfrak{gl}(n)$. Therefore we need to assume that the basis $\{X_a\}$ of $\mn$ extends with a basis $\{T_\kappa\}$ of $\Mp^\ast$ to a basis of $\mg$ for which the dual basis with respect to the (normalized) Killing form is of the form $\{X_a^\dagger, T_\kappa^\ddagger\}$. Note that this is true for the cases considered in Theorem \ref{tensor} because of Remark \ref{Killing}. The necessity of introducing the notation $\ddagger$ comes from the fact that this mapping will not correspond to the star map $\dagger$ on $\mg$ corresponding to the representation $\mV$.

\begin{theorem}
\label{Thmquabla}
Consider a Lie superalgebra $\mg$, with parabolic subalgebra $\Mp$ and representation $\mV$ for which conditions \eqref{neccond} are satisfied. Assume there exists a basis $\{ X_a, T_\kappa\}$ of $\mg$ such that the dual basis $\{X^{\ddagger},T_\kappa^\ddagger\}$ with respect to the normalized Killing form $(\cdot,\cdot)_K$satisfies $X_a^\ddagger=X_a^\dagger$, then the Kostant quabla operator on $C^k(\overline{\mn},\mV)$ takes the form
\begin{eqnarray*}
\iota(\Box f)&=&-\frac{k+1}{2}\left(\cC_2(\mV)+\sum_a X_a X_a^\dagger-\sum_\kappa T_\kappa T_\kappa^\ddagger\right) \iota(f)\\
&=&-\frac{k+1}{2}\left(\cC_2(\mV)-\cC_2+2\sum_a X_a X_a^\dagger\right) \iota(f)
\end{eqnarray*}
where $\iota$ is the natural embedding $\iota:\Lambda^k\mn\otimes\mV\,\hookrightarrow \,\Lambda^k\mg\otimes\mV$ and $\cC_2(\mV)$ the value of the Casimir operator $\cC_2=\sum_aX_aX^\dagger_a +\sum_\kappa T_\kappa T^\ddagger_\kappa$ on $\mV$.
\end{theorem}
Before we prove this relation, we note that even though the right hand side is defined as an element of $\Lambda^k\mg\otimes\mV$ it follows from the theorem that it is contained in $C^k(\overline{\mn},\mV)$.
\begin{proof}
We start by observing that the operator $\partial:C^k(\overline{\mn},\mV)\to C^{k+1}(\overline{\mn},\mV)$ can be rewritten as
\begin{eqnarray*}
\iota\left(\partial (X^1\wedge\cdots\wedge X^k\otimes v)\right)&=&\frac{k+1}{2}\sum_aX_a\wedge [X^\dagger_a,X^1\wedge\cdots\wedge X^k]\otimes v\\
&&-\frac{k+1}{2}\sum_\kappa T_\kappa\wedge [T^\ddagger_\kappa,X^1\wedge\cdots\wedge X^k]\otimes v\\
&&+(k+1)\sum_a(-1)^{|X_a|(|X^1|+\cdots+|X^k|)} X_a\wedge X^1\wedge\cdots\wedge X^k\otimes X^\dagger_a\cdot v
\end{eqnarray*}
which follows from the proof of Lemma \ref{propstandder} and the relation
\begin{eqnarray*}
\sum_a X_a\wedge [X_a^\dagger,X]_{\Mp^\ast}&=&\sum_a X_a\wedge \sum_\kappa \left(T_\kappa^{\ddagger},[X_a^\dagger,X]\right)T_\kappa\\
&=&-\sum_\kappa (-1)^{|X|(|X|+|T_\kappa|)}\sum_a X_a\wedge  \left([T_\kappa^{\ddagger},X],X_a^\dagger\right)T_\kappa\\
&=&-\sum_\kappa (-1)^{|T_\kappa|(|X|+|T_\kappa|)}\sum_a X_a\wedge  \left(X_a^\dagger,[T_\kappa^{\ddagger},X]\right)T_\kappa=T_\kappa\wedge [T_\kappa^{\ddagger},X]
\end{eqnarray*}
for $X\in\mn$. The property 
\begin{eqnarray*}
&&\sum_a\left( (-1)^{|X_a||X|}[X_a,X]\otimes X^\dagger_a\cdot v+ X_a\otimes [X_a^\dagger,X]\cdot v\right)\\
&=&-\sum_{\kappa}\left((-1)^{|T_\kappa||X|}[T_\kappa,X]\otimes T_\kappa^\ddagger\cdot v+ T_\kappa\otimes [T_\kappa^\ddagger,X]\cdot v\right)
\end{eqnarray*}
follows from the invariance of the Killing form. The proposed equality can then be calculated by using this identity, along the lines of \cite{MR0142696}.
\end{proof}

The space $C^k(\mn,\mV^\ast)=\Lambda^k\overline{\mn}\otimes \mV^\ast$ can be identified with the dual space of $C^k(\overline{\mn},\mV)=\Lambda^k\mn\otimes \mV$ where the pairing is induced from the bilinear form in equation \eqref{neccond} and the pairing between $\mV$ and $\mV^\ast$ is as described at the beginning of Section \ref{ExamGL}. 

\begin{definition}
The bilinear form 
\begin{eqnarray*}
(\cdot,\cdot):\left(\otimes^k\overline{\mn}\otimes\mV^\ast\right)\times\left(\otimes^k\mn\otimes\mV\right)\to\mC
\end{eqnarray*}
is defined inductively by 
\begin{eqnarray*}
(Y\otimes q,X\otimes p)&=&(-1)^{|q||X|}(Y,X)(q,p)
\end{eqnarray*}
for $Y\in\overline{\mn}$, $X\in\mn$, $q\in\otimes^{k-1}\overline{\mn}\otimes\mV^\ast$, $p\in\otimes^{k-1}{\mn}\otimes\mV$ and with the bilinear form $\overline{\mn}\times\mn\to\mC$ from properties \eqref{neccond} and the bilinear form $\mV^\ast\times\mV\to\mC$ corresponding to the evaluation of $\mV^\ast$ on $\mV$. The bilinear form $C^k(\mn,\mV^\ast)\times C^k(\overline{\mn},\mV)\to\mC$ is the restriction of the form above.
\end{definition}
It is important to pay attention to the powers of $-1$ which are included in this definition and which did not appear in products of inner products, e.g. in equation \eqref{prodinnerprod}. This bilinear form satisfies 
\[(A[q],p)=-(-1)^{|A||q|}(q,A[p])\]
with $A[p]$ defined as above Lemma \ref{conscond} for $A\in\mg$ and $A[q]$ analogously.

Then we can define the operators $\delta^\ast: C^k(\mn,\mV^\ast)\to C^{k-1}(\mn,\mV^\ast)$ and $\delta:C^k(\mn,\mV^\ast)\to C^{k+1}(\mn,\mV^\ast)$ as the adjoint operators with respect to $(\cdot,\cdot)$ of respectively $\partial$ and $\partial^\ast$. By the structure of the bilinear form they are $\mg_0$-module morphisms. Direct calculations lead to the result that $\delta^\ast$ acts as the codifferential. It is insightful to mention that the operator $\delta$ has properties as
\begin{eqnarray*}
\delta(w)&=&\sum_a(-1)^{|X_a|}X^\dagger_a\otimes X_a\cdot w
\end{eqnarray*}
for $w\in\mV^\ast$. This is logical in the sense of Proposition \ref{stardual} since the star map on $\mg$ corresponding to $\mV^\ast$ is exactly given by $A\to(-1)^{|A|}A^\dagger$. All these considerations lead to the property in the following lemma.
\begin{lemma}
\label{dualcohom}
The cohomology space $H^k(\mn,\mV^\ast)$ defined as the $\mg_0$-module {\rm $\ker \delta^\ast_k/\mbox{im}\delta^\ast_{k+1}$} satisfies
\begin{eqnarray*}
H^k(\mn,\mV^\ast)&=&\left(H^k(\overline{\mn},\mV)\right)^\ast
\end{eqnarray*}
with $\left(H^k(\overline{\mn},\mV)\right)^\ast$ the $\mg_0$-dual module of $H^k(\overline{\mn},\mV)$. By seeing $H^k(\mn,\mV^\ast)$ as a $\Mp$-module with trivial $\mn$-action this duality also holds as $\Mp$-modules.
\end{lemma}

Finally we prove how the cohomology groups can be easily characterized using the quadratic Casimir operator.
\begin{theorem}
\label{structure cohomgroups}
Consider $\mg=\mathfrak{gl}(m|n)$, a parabolic supalgebra $\Mp=\mg_0+\mn$ and a representation $\mW_\lambda$ such that $(\mg,\Mp,\mW_\lambda)$ satisfies conditions \eqref{neccond2}. As a $\mg_0$-module, the Kostant cohomology group $H^k(\mn,\mW_\lambda)$ is the direct sum of the $\mg_0$-submodules of $\Lambda^k\overline{\mn}\otimes \mW_\lambda$ which have a highest weight $\mu$ such that $\cC_2(\mW_\lambda)=\cC_2(\mW_\mu)$.
\end{theorem}
\begin{proof}
The Kostant quabla operator on $\Lambda^k\overline{\mn}\otimes \mW_\lambda$ takes the dual form of the one in Theorem \ref{Thmquabla}. Therefore it can quickly be seen that this operator commutes with the action of $\mg_0$ and therefore acts as a scalar on an irreducible $\mg_0$-submodule of $\Lambda^k\overline{\mn}\otimes \mW_\lambda$. The action of $\Box$ on a highest weight vector of $\mg_0$ is given by the value $\cC_2(\mW_\lambda)-\cC_2(\mW_\mu)$.
\end{proof}
In the specific case $\mg_0=\mathfrak{gl}(m)\oplus\mathfrak{gl}(n)$, this was proved in Lemma 4.5 in \cite{MR2036954}.

\section{The BGG resolutions}
\label{BGG}
We continue to use the same notations as in the previous section and the introduction, with the assumption that the properties \eqref{neccond} are satisfied.
\begin{definition}
\label{defjet}
\rm For a $\Mp$-module $\mF$, the coinduced $\mg$-module $\cJ(\mF)$ is defined as a vector space by 
\begin{eqnarray*}
\mbox{Hom}_{\cU(\Mp)}(\cU(\mg),\mF)&=&\{\alpha\in \mbox{Hom}(\cU(\mg),\mF)|\alpha(U Z)=-(-1)^{|U||Z|}Z\left(\alpha(U)\right)\quad\mbox{for }Z\in\Mp\mbox{ and }U\in\cU(\mg)\}.
\end{eqnarray*}
The action of $\mg$ on $\cJ(\mF)$ is defined as $(A\alpha)(U)=-(-1)^{|A||U|}\alpha(AU)$ for $\alpha\in \cJ(\mF)$, $A\in\mg$ and $U\in\cU(\mg)$, which makes $\cJ(\mF)$ a $\mg$-submodule of Hom$(\cU(\mg),\mF)$.
 \end{definition}

We will need these spaces for the case $\mF=C^k(\overline{\mn},\mV)$ and submodules and subquotients. Note that it is important to write the brackets, for instance, with $Z\in\Mp$, $U\in\cU(\mg)$ and $\alpha\in\cJ\left(\mF\right)$, 
\[(Z\alpha)(U)=-(-1)^{|Z||U|}\alpha(ZU)\,\mbox{ is not the same as }\, Z(\alpha(U))=-(-1)^{|Z||U|}\alpha(UZ).\]

Since the operator $\partial^\ast$ is $\Mp$-invariant, it immediately extends to a $\mg$-invariant operator $\cJ\left(C^k(\overline{\mn},\mV)\right)\to \cJ\left(C^{k-1}(\overline{\mn},\mV)\right)$. The operator $\partial$ is not $\Mp$-invariant, see Corollary \ref{pactionder}, but it can be modified to an operator $d:\cJ\left(C^k(\overline{\mn},\mV)\right)\to \cJ\left(C^{k+1}(\overline{\mn},\mV)\right)$ which is $\mg$-invariant. This is the subject of the following theorem. In the language of differential operators on parabolic geometries, this corresponds to constructing a morphism between the first jet extension of $C^k(\overline{\mn},\mV)$ and $C^{k+1}(\overline{\mn},\mV)$, with scalar part equal to $\partial$. This is a twisted exterior derivative, as will be discussed in the Appendix.

\begin{theorem}
\label{thmdefop}
The operators
\begin{eqnarray*}
\partial^\ast:\cJ\left(C^k(\overline{\mn},\mV)\right)\to \cJ\left(C^{k-1}(\overline{\mn},\mV)\right)&& \left(\partial^\ast\alpha\right)(U)=\partial^\ast\left(\alpha(U)\right)\\
d:\cJ\left(C^k(\overline{\mn},\mV)\right)\to \cJ\left(C^{k+1}(\overline{\mn},\mV)\right)&&\left(d\alpha\right)(U)=\partial \left(\alpha(U)\right)+(k+1)\sum_a(-1)^{|U||X_a|} X_a\wedge \alpha(UX^\dagger_a)
\end{eqnarray*}
for $\alpha\in \cJ\left(C^k(\overline{\mn},\mV)\right)$ and $U\in\cU(\mg)$, with action of $\partial$ and $\partial^\ast$ on $C^k(\overline{\mn},\mV)$ as defined in Section \ref{Kostant} are $\mg$-module morphisms.
\end{theorem}
\begin{proof}
The operator $\partial^\ast:\cJ\left(C^k(\overline{\mn},\mV)\right)\to \mbox{Hom}(\cU(\mg),C^{k-1}(\overline{\mn},\mV))$ is clearly a $\mg$-module morphism. It remains to be checked that $\mbox{im}(\partial^\ast)\subset \cJ\left(C^{k-1}(\overline{\mn},\mV)\right)$. This follows from the fact that for $Z\in\Mp$ and $U\in\cU(\mg)$ it holds that
\begin{eqnarray*}&&(\partial^\ast\alpha)(UZ)=\partial^\ast\left(\alpha(UZ)\right)=-(-1)^{|U||Z|}\partial^\ast\left(Z\cdot\alpha(U)\right)\\
&=&-(-1)^{|U||Z|}Z\cdot\partial^\ast\left(\alpha(U)\right)=-(-1)^{|U||Z|}Z\cdot\left(\left(\partial^\ast\alpha\right)(U)\right).
\end{eqnarray*}

The fact that the operator $d:\cJ\left(C^k(\overline{\mn},\mV)\right)\to \mbox{Hom}\left(\cU(\mg),\left(C^{k+1}(\overline{\mn},\mV)\right)\right)$ is $\mg$-invariant is also trivial. It is the sum of two $\mg$-invariant operators acting between $\cJ\left(C^k(\overline{\mn},\mV)\right)$ and $\mbox{Hom}(\cU(\mg),C^{k+1}(\overline{\mn},\mV))$. We define the operator $T$ as the second term in the definition of $d$, so $d=\partial+T$.

We need to compute that the image of $d$ is contained in $ \cJ\left(C^{k+1}(\overline{\mn},\mV)\right)$. It follows from Definition \ref{defjet} and equation \eqref{propdualbasis} that
\begin{eqnarray*}
\frac{1}{k+1}(T\alpha)(UZ)&=&\sum_a(-1)^{|Z||U|} X_a\wedge[X^\dagger_a,Z]_\Mp\cdot \left(\alpha(U)\right)-\sum_a (-1)^{(|X_a|+|Z|)|U|}[Z,X_a]\wedge \alpha(UX^\dagger_a)\\
&&-\sum_a(-1)^{|U|(|X_a|+|Z|)+|Z||X_a|} X_a\wedge Z\cdot(\alpha(UX^\dagger_a))\\
&=&\sum_a(-1)^{|Z||U|} X_a\wedge[X^\dagger_a,Z]_\Mp\cdot \left(\alpha(U)\right)- (-1)^{|U||Z|}\frac{1}{k+1}Z\left((T\alpha)(U)\right)
\end{eqnarray*}
holds. Comparison with Proposition \ref{pactionder} then yields the proof.
\end{proof}

This operator $d$ generates a coresolution of the module $\mV$, this is the subject of the following theorem.
\begin{theorem}
\label{Rhamcomplex}
The sequence
\begin{eqnarray*}
0\to\mV\to^{\epsilon} \cJ(\mV)\to^{d_0}\cJ(C^1(\overline{\mn},\mV))\to^{d_1}\cdots\to^{d_k}\cJ(C^{k+1}(\overline{\mn},\mV))\to^{d_{k+1}}\cdots
\end{eqnarray*}
with $\epsilon$ defined by $\left(\epsilon(v)\right)(U)=S(U)v$ with $S$ the principal anti-automorphism of $\cU(\mg)$ (the antipode as described in the Appendix) is a coresolution of $\mV$.
\end{theorem}
\begin{proof}
It needs to be proven that the sequence is a complex, $d_k\circ d_{k-1}=0$ and moreover that this complex is exact, ker$d_k=$im$d_{k-1}$.

First we prove that im$\epsilon=\ker d_0$. From the definition of $\epsilon(v)$ and Lemma \ref{propstandder} it follows that $d_0\circ\epsilon=0$. Now if $\alpha\in \cJ(\mV)$ is in the kernel of $d_0$, then if follows quickly that $\alpha(UY)=-(-1)^{|U||Y|}Y( \alpha(U))$ for $Y\in \overline{\mn}$. Together with Definition \ref{defjet} this implies that $\alpha(UA)=-(-1)^{|U||A|}A( \alpha(U))$ for all $A\in\mg$ or $\alpha(U)=S(U)\left(\alpha(1)\right)$, so im$\epsilon\cong\mV\cong \ker d_0$.

The fact that $d_k\circ d_{k-1}=0$ follows from a direct calculation using Lemma \ref{propstandder}, for $\alpha\in \cJ(C^{k-1}(\overline{\mn},\mV))$
\begin{eqnarray*}
(dd\alpha)(U)&=&\partial \left(d\alpha(U)\right)+(k+1)\sum_a(-1)^{|U||X_a|} X_a\wedge (d\alpha)(UX^\dagger_a)\\
&=&k\sum_a(-1)^{|U||X_a|}\partial \left(X_a\wedge \alpha(UX^\dagger_a)\right)+(k+1)\sum_a(-1)^{|U||X_a|} X_a\wedge \partial \alpha(UX^\dagger_a)\\
&&+k(k+1)\sum_{a,b}(-1)^{|U|(|X_a|+|X_b|)+|X_a||X_b|} X_a\wedge X_b\wedge \alpha(UX^\dagger_aX_b^\dagger)\\
&=&\frac{k(k+1)}{2} \sum_{a,b}(-1)^{|X_a||U|}X_b\wedge [X_b^\dagger,X_a]_{\mn}\wedge \alpha(UX^\dagger_a)\\
&&+\frac{k(k+1)}{2}\sum_{a,b}(-1)^{|U|(|X_a|+|X_b|)+|X_a||X_b|} X_a\wedge X_b\wedge \alpha(U[X^\dagger_a,X_b^\dagger]),
\end{eqnarray*}
which is zero by equation \eqref{propdualbasis}.

Because of Theorem \ref{twisted} in the Appendix the exactness can be reduced to the case where $\mV=0$. Through the vector space isomorphism $\cJ(\Lambda^k\mn)\cong$ Hom$(\cU(\overline{\mn}),\Lambda^k\mb)$ obtained from the Poincar\'e-Birkhoff-Witt property, the exactness of $d$ corresponds to the exactness of its induced operator Hom$(\cU(\overline{\mn}),\Lambda^k\mb)\to$ Hom$(\cU(\overline{\mn}),\Lambda^{k+1}\mb)$. Theorem \ref{extder} in the Appendix shows that the homology of the operator $d$ then becomes equivalent with that of the formal exterior derivative on a flat supermanifold. This is known to have trivial homology, see \cite{MR0672426}.
\end{proof}

At this stage we need to point out that $C^k(\overline{\mn},\mV)$ has a finite filtration as a $\Mp$-module. Since it is a star $\mg_0$-module, it decomposes into irreducible $\mg_0$-modules. The corresponding gradation follows from giving each simple root vector in $\mn$ degree one and the elements of $\mg_0$ degree zero. This gradation is naturally inherited by Hom$(\cU(\mg),C^k(\overline{\mn},\mV))$ where the gradation of a homomorphism is given by the gradation of its values. It will very useful for the sequel that the operators $\partial^\ast$ and $\partial$ have degree zero, while the operator $T$ raises the degree by one. The $\mg$-invariant operator related to the Kostant quabla operator is given by
\begin{eqnarray*}
\widetilde{\Box}&=&d\partial^\ast+\partial^\ast d=\Box+T\partial^\ast+\partial^\ast T \,:\,\, \cJ\left(C^k(\overline{\mn},\mV)\right)\to \cJ\left(C^k(\overline{\mn},\mV)\right)
\end{eqnarray*}
Since the operator $\Box_T=\Box-\widetilde{\Box}=-T\partial^\ast-\partial^\ast T$ raises degree on $\mbox{Hom}(\cU(\mg),C^k(\overline{\mn},\mV))$ it is nilpotent, this fact will be essential. Obviously also compositions of $\Box_T$ with operators of degree 0 are nilpotent.


\begin{lemma}
\label{inversequabla}
The operator $\widetilde{\Box}$ is invertible on {\rm $\mbox{im}\partial^\ast$}.
\end{lemma}
\begin{proof}
Lemma \ref{lemmaKostant} implies that $\Box$ is invertible on im$\Box$. Then it follows quickly that the finite sum
\begin{eqnarray*}
\sum_{j\ge 0}(\Box_T \Box^{-1})^{j}\Box^{-1}
\end{eqnarray*}
is the inverse of $\widetilde{\Box}$ on im$\Box$, so also on im$\partial^\ast$, by Lemma \ref{lemmaKostant}.
\end{proof}

\begin{definition}
The $\mg$-invariant operators $\Pi_k:\cJ(C^k(\overline{\mn},\mV))\to \cJ(C^k(\overline{\mn},\mV))$ are defined as
\begin{eqnarray*}
\Pi_k=1\,-\,\,d_{k-1}\circ \widetilde{\Box}^{-1}\circ \partial^\ast_k\,\,-\,\,\widetilde{\Box}^{-1}\circ \partial^\ast_{k+1}\circ d_k.
\end{eqnarray*}
\end{definition}
This is well-defined by Lemma \ref{inversequabla}. We introduce the notation $p_k$ for the $\mg$-invariant projection $\cJ(\mbox{ker}\partial^\ast_k)\to \cJ(H^k(\overline{\mn},\mV))$. Note that $\cJ(\mbox{ker}\partial^\ast_{k})$ is identical to the kernel of $\partial^\ast_{k}$ on $\cJ(C^{k}(\overline{\mn},\mV))$ and likewise for the image.

We also use the notation $repr_k$ for any $\mg$-module morphism from $\cJ(H^k(\overline{\mn},\mV))$ to $\cJ(\mbox{ker}\partial^\ast_k)$ that is inverted by $p_k$. The following properties of the operators $\Pi_k$ are now straightforward to derive.
\begin{lemma}
\label{propsplit}
The $\mg$-invariant operators $\Pi_k:\cJ(C^k(\overline{\mn},\mV))\to \cJ(C^k(\overline{\mn},\mV))$ satisfy
\begin{eqnarray*}
(1)\quad \Pi_k\circ\partial^\ast_{k+1}=0=\partial^\ast_k\circ\Pi_k&&(3)\quad \Pi_{k+1}\circ d_{k}=d_k\circ\Pi_k \\
(2)\qquad \quad\, p_k\circ \Pi_k\circ repr_k=1&&(4)\quad \Pi_k^2=\Pi_k.
\end{eqnarray*}
\end{lemma}

Since the composition $\Pi_k\circ repr_k$ does not depend on the choice of $repr_k$ by Lemma \ref{propsplit} $(1)$ we fix the notation 
\[L_k=\Pi_k\circ repr_k\,:\,\, \cJ(H^k(\overline{\mn},\mV))\to\cJ(\ker\partial^\ast).\]
It follows from lemma \ref{propsplit} $(2)$ that $p_k\circ L_k=1$, so in particular the $\mg$-invariant operator $L_k$ is injective.

Now we come to the definition of the operator that will be responsible for the desired coresolution of $\mV$.
\begin{definition}
\label{defD}
The $\mg$-invariant operator $D_k: \cJ(H^k(\overline{\mn},\mV))\to  \cJ(H^{k+1}(\overline{\mn},\mV))$ is defined as
\begin{eqnarray*}
D_k&=&p_{k+1}\circ d_k\circ L_k\,=\,p_{k+1}\circ \Pi_{k+1}\circ d_k\circ \Pi_k\circ repr_k\,=\,p_{k+1}\circ \Pi_{k+1}\circ d_k\circ repr_k.
\end{eqnarray*}
\end{definition}
The proposed formulas are equivalent by Lemma \ref{propsplit} $(3)$ and $(4)$. From the definition it immediately follows that $L_{k+1}\circ D_k=d_k\circ L_k$.

Using the defined operators it is now possible to construct a smaller sequence out of the twisted de Rham sequence, that corresponds to the infinitesimal version of the classical dual BGG sequences, mentioned in Section \ref{classical}.
\begin{theorem}
\label{BGGseq}
The sequence
\begin{eqnarray*}
0\to\mV\to^{\epsilon'} \cJ(H^0(\overline{\mn},\mV))\to^{D_0}\cJ(H^1(\overline{\mn},\mV))\to^{D_1}\cdots\to^{D_k}\cJ(H^{k+1}(\overline{\mn},\mV))\to^{D_{k+1}}\cdots
\end{eqnarray*}
with $\epsilon'$ defined as $\epsilon' (v)=p_0(\epsilon(v))$ and $\epsilon$ given in Theorem \ref{Rhamcomplex} is a coresolution of $\mV$.
\end{theorem}
\begin{proof}
First we prove that this forms a complex,
\begin{eqnarray*}
D_{k+1}\circ D_{k}&=&p_{k+2}\circ \Pi_{k+2}\circ d_{k+1}\circ \Pi_{k+1}\circ repr_{k+1}\circ p_{k+1}\circ \Pi_{k+1}\circ d_k\circ \Pi_k\circ repr_k\\
&=&p_{k+2}\circ \Pi_{k+2}\circ d_{k+1}\circ \Pi_{k+1}\circ  \Pi_{k+1}\circ d_k\circ \Pi_k\circ repr_k\\
&=&p_{k+2}\circ \Pi_{k+2}\circ d_{k+1}\circ d_k\circ \Pi_k\circ repr_k=0
\end{eqnarray*}
where we have used the fact that $\Pi_{k+1}$ does not see which representative is chosen by Lemma \ref{propsplit} $(1)$, then Lemma \ref{propsplit} $(4)$ and $(3)$ and finally Theorem \ref{Rhamcomplex}. The fact $D_0\circ\epsilon'=0$ follows similarly.

Now we prove the exactness of the complex. The fact that im$\epsilon'=\ker D_0$ follows from the property that $\epsilon'$ is injective (since $\mV$ is irreducible) and the fact that $L_0$ maps the kernel of $D_0$ injectively into the kernel of $d_0$ which is isomorphic to $\mV$, see Theorem \ref{Rhamcomplex}.

We consider an $f\in \cJ(H^{k+1}(\overline{\mn},\mV))$ that satisfies $D_{k+1}f=0$. Since $L_{k+2}\circ D_{k+1}=d_{k+1}\circ L_{k+1}$ it follows that $L_{k+1}(f)= d_{k}g$ for some $g\in \cJ(C^{k}(\overline{\mn},\mV))$ by Theorem \ref{Rhamcomplex}.

It can be proven that there is some $h\in \cJ(C^{k-1}(\overline{\mn},\mV))$, such that $g+d_{k-1}h$ is inside $\cJ(\ker\partial^\ast)$ by using Lemma \ref{lemmaKostant} which implies $C^k(\overline{\mn},\mV)=\mbox{im}\partial_k\oplus \ker \partial^\ast_k$ and the fact that $d-\partial=T$ strictly raises degree. Therefore we obtain $L_{k+1}(f)= d_{k}g'$ for $g'\in\cJ(\ker\partial^\ast)$.

Since $d_{k}g'=L_{k+1}(f)$, $d_kg'$ is also inside $\cJ(\ker\partial^\ast)$. We can prove that the relation $g'=L_{k}\circ p_{k}(g')$ holds as follows. The element $g'-L_{k}\circ p_{k}(g')$ of $\cJ(\ker\partial^\ast)$ is inside $\cJ($im$\partial^\ast)$ since its projection onto $\cJ(H^{k}(\overline{\mn},\mV))$ is zero by the relation $p_{k}\circ L_k=1$. Because $\Box$ is invertible on $\mbox{im}\partial^\ast$ by Lemma \ref{lemmaKostant}, we can prove $g'-L_{k}\circ p_{k}(g')=0$ by calculating
\begin{eqnarray*}
\Box(g'-L_{k}\circ p_{k}(g'))&=&(\partial^\ast d_k-\partial^\ast T_k)(g'-L_{k}\circ p_{k}(g'))\\
&=&\partial^\ast L_{k+1}(f)-\partial^\ast L_{k+1}\circ D_k\circ p_{k}(g')-\partial^\ast T_k(g'-L_{k}\circ p_{k}(g'))\\
&=&-\partial^\ast T_k(g'-L_{k}\circ p_{k}(g')).
\end{eqnarray*}
The operator $\Box$ is of degree zero while its action above strictly raises the degree, so the only possible option is that the action of this invertible operator is zero. This shows that $g'=L_{k}\circ p_{k}(g')$ holds.

Therefore we obtain that if $D_{k+1}f=0$ holds for $f\in \cJ(H^{k+1}(\overline{\mn},\mV))$, then $L_{k+1}(f)=L_{k+1}\circ D_k \circ p_k(g')$ holds. By injectivity of $L_{k+1}$ it follows that $f\in$ im$D_{k}$ and the theorem is proven.
\end{proof}

This leads to the main theorem of this paper. We use the notation $V^\mF$ for any $\mg$-module induced from a $\Mp$-module $\mF$. We only use the terminology (generalized) Verma module in case $\mF$ is irreducible.
\begin{theorem}
\label{finalthm}
For a semisimple Lie superalgebra $\mg$ with irreducible finite dimensional representation $\mW$ and a parabolic subalgebra $\Mp=\mg_0+\mn$, such that the triple $(\mg,\Mp,\mW)$ satisfies conditions \eqref{neccond2}, the module $\mW$ has a resolution in terms of generalized Verma modules given by
\begin{eqnarray*}
\cdots \to V^{H^k(\mn,\mW)}\to\cdots\to V^{H^1(\mn,\mW)}\to V^{H^0(\mn,\mW)}\to\mW\to0
\end{eqnarray*}
with Kostant cohomology group {\rm $H^k(\mn,\mW)=\ker\delta^\ast_{k}/\mbox{im}\delta^\ast_{k+1}$}, which is a finite dimensional completely reducible $\Mp$-module (described in Theorem \ref{structure cohomgroups}), and $V^{H^k(\mn,\mW)}=\cU(\mg)\otimes_\Mp H^k(\mn,\mW)$.
\end{theorem}

\begin{proof}
This is a consequence of Theorem \ref{BGGseq} and the fact that there exists a non-degenerate bilinear $\mg$-invariant paring between 
\[V^{\mF^\ast}=\cU(\mg)\otimes_{\Mp}\mF^\ast \mbox{ and } \cJ(\mF)=\mbox{Hom}_{\cU(\Mp)}(\cU(\mg),\mF)\]
for any irreducible $\Mp$-representation $\mF$.

This pairing is induced by the $\mg$-invariant pairing between $\cU(\mg)\otimes\mF^\ast$ and $\mbox{Hom}(\cU(\mg),\mF)$ given by
\begin{eqnarray*}
(U\otimes\alpha,\phi)=(-1)^{|U||\alpha|}\alpha\left(\phi(U)\right)&\mbox{for}&U\in\cU(\mg), \,\alpha\in\mF^\ast\mbox{ and }\phi\in\mbox{Hom}(\cU(\mg),\mF).
\end{eqnarray*}
This pairing is clearly non-degenerate and in this context $\mg$-invariant means $(Au,\phi)=-(-1)^{|A||u|}(u,A\phi)$ for $A\in\mg$, $u\in\cU(\mg)\otimes\mF^\ast$ and $\phi\in\mbox{Hom}(\cU(\mg),\mF)$. Now we restrict the paring to the two subrepresentations we are interested in. We show that this pairing is still non-degenerate. For every $\phi\in \cJ(\mF)$ there is an element of $\cU(\mg)\otimes\mF^\ast$ which has a non-zero evaluation on $\phi$. By first restricting the element to $\cU(\overline{\mn})\otimes\mF^\ast$ and seeing it as an element of $V^{\mF^\ast}$ it follows that this element has the same evaluation on $\cJ(\mF)$ as the original one. The proof of the non-degeneracy in the other direction is similar. From this non-degenerate pairing it follows that if two subspaces of $V^{\mF^\ast}$ have the same space of orthogonal vectors inside $\cJ(\mF)$ they must coincide. 

Since the triple $(\mg,\Mp,\mW)$ satisfies conditions \eqref{neccond2}, the triple $(\mg,\Mp,\mV)$ with $\mV=\mW^\ast$ satisfies conditions \eqref{neccond}, so we can apply Theorem \ref{BGGseq}.

The $\mg$-invariant operators 
\[D^\ast_k:\mV^{H^k(\overline{\mn},\mV)^\ast}\to\mV^{H^{k-1}(\overline{\mn},\mV)^\ast}\]
defined by $(D_k^\ast u,\phi)=(u,D_k\phi)$ with $D_k$ from Definition \ref{defD} for all $u\in\mV^{H^k(\overline{\mn},\mV)^\ast}$ and $\phi\in \cJ(H^k(\overline{\mn},\mV))$ then form an exact complex by the considerations above since for instance $\ker D^\ast_k$ corresponds to space of orthogonal vectors of im$D_k$.

The proof then follows from the identification $H^j(\mn,\mV^\ast)\cong H^j(\overline{\mn},\mV)^\ast$ in Lemma \ref{dualcohom} and Corollary \ref{Hcompred}.
\end{proof}

Contrary to the classical case of Lepowsky the resolution is not finite since the homolgy groups do not vanish (because there is no top power for super anti-symmetric tensors). The method of super-duality gives a nice interpretation of this infinite behavior by identifying the weights of the highest weight vectors in $H^j(\mn,\mW_\lambda)$ with orbits of the parabolic Weyl group of the Lie algebra $\mathfrak{gl}(m+\infty)$, see e.g. \cite{Cheng, MR2036954}. If one ignores the superduality with $\mathfrak{gl}(m+\infty)$ and only looks at the Lie superalgebra $\mathfrak{gl}(m|n)$ the highest weights come from both reflections of the parabolic Weyl group of $\mathfrak{gl}(m)\oplus\mathfrak{gl}(n)$ as well as from so-called odd reflections. 

As a consequence of the obtained BGG resolutions we have the following corollary which is a non-trivial observation in the case of Lie superalgebras.
\begin{corollary}
\label{submodVerma}
Consider $\mg=\mathfrak{gl}(m|n)$ and $\Mp$ a parabolic subalgebra that contains $\mathfrak{gl}(n)$. For $\lambda\in\mathfrak{h}^\ast$ the highest weight of a tensor module of $\mathfrak{gl}(m|n)$, define $M(\lambda)$ as the $\mg_0$-module with highest weight $\lambda$ which is a $\Mp$-module with trivial $\mn$-action. The generalized Verma module
\begin{eqnarray*}
V^{M(\lambda)}&=&\cU(\mg)\otimes_{\Mp}M(\lambda)
\end{eqnarray*}
satisfies the property that its unique maximal submodule is generated by highest weight vectors in $V^{M(\lambda)}$.
\end{corollary}
\begin{proof}
This follows from Theorem \ref{tensor} and the last part of the BGG resolution in Theorem \ref{finalthm}. The sequence
\begin{eqnarray*}
V^{H^1(\mn,\mV_\lambda)}\to V^{M(\lambda)}\to\mV_\lambda\to0
\end{eqnarray*}
is exact, where the property $M(\lambda)\cong H^0(\mn,\mV_\lambda)$ follows straight away. This implies that $\mV_\lambda\cong  V^{M(\lambda)}/N$ with $N$ a subquotient of $V^{H^1(\mn,\mV_\lambda)}$, which is always a module generated by highest weight vectors.

\end{proof}
In particular this states that for $\mathfrak{gl}(m|1)$, ordinary Verma modules, with a tensor module highest weight, possess this property. It is obvious that similar statements can be made for duals of tensor modules of $\mathfrak{gl}(m|n)$ and parabolics containing $\mathfrak{gl}(m)$.

\section*{Appendix: The Hopf superalgebra $\cU(\mg)$ and the twisted de Rham operator}

In this appendix we show how the operator $d:\cJ(C^k(\overline{\mn},\mV))\to \cJ(C^{k+1}(\overline{\mn},\mV))$ from Theorem \ref{thmdefop} is obtained from the same operator for the case $\mV=0$ by a $\mg$-module isomorphism between $\cJ(C^k(\overline{\mn},\mV))$ and $\cJ(\wedge^k{\mn})\otimes\mV$. In the second part of the appendix we prove that the operator $d$ for the case $\mV=0$ can be rewritten as a standard exterior derivative. Therefore we can interpret the operator $d:\cJ(C^k(\overline{\mn},\mV))\to \cJ(C^{k+1}(\overline{\mn},\mV))$ as a twisted de Rham operator. These results are necessary to prove the exactness of the operator $d$, which is equivalent with the exactness of the BGG complex as is proved in Section \ref{BGG}.

The universal enveloping algebra of a Lie superalgebra has the structure of a supercocommutative Hopf superalgebra, see \cite{MR0252485}. The comultiplication $\Delta:\cU(\mg)\to\cU(\mg)\otimes\cU(\mg)$, antipode $S:\cU(\mg)\to\cU(\mg)$, multiplication $m:\cU(\mg)\otimes\cU(\mg)\to\cU(\mg)$ and co-unit $\varepsilon:\cU(\mg)\to\mC$ are generated by
\begin{eqnarray*}
\Delta(A)=A\otimes 1+1\otimes A&&S(A)=-A\\
m(A\otimes B)=AB&&\varepsilon(A)=0
\end{eqnarray*}
for $A,B\in\mg$. Basic properties that we will need are
\begin{eqnarray*}
m\circ\left(S\otimes 1\right)\circ\Delta&=&\varepsilon\qquad\mbox{and}\\
(\Delta\otimes 1)\circ\Delta&=&(1\otimes \Delta)\circ\Delta.
\end{eqnarray*}
We use the Sweedler notation $\Delta (U)=\sum_j U^j_1\otimes U^j_2$ for $U\in\cU(\mg)$.

In order to describe the morphism structure between $\cJ(C^k(\overline{\mn},\mV))$ and $\cJ(\wedge^k{\mn})\otimes\mV$ properly we need some extra notations. For every $A\in\mg$ the action of $A^\mV$ on $C^k(\overline{\mn},\mV)$ is given by the action on $\mV$ and trivial action on $\wedge^k\mn$, only the gradation needs to be taken into account. This notation extends to $U^\mV$ for $U\in\cU(\mg)$. Likewise we define the action $A^\cJ$ on $C^k(\overline{\mn},\mV)$ as $A$ acting on $\wedge^k\mn$ and regarding $\mV$ as a sum of trivial representations. In particular $A^\mV+A^\cJ$ gives the ordinary action of $\mg$ on $C^k(\overline{\mn},\mV)$ in equation \eqref{gactionchain}. As discussed before, this is not a representation of $\mg$ on $C^k(\overline{\mn},\mV)$, it can only be identified with representations when restricted to the subalgebras $\Mp$ or $\Mp^\ast$. Note that the $\mg$-module structure of $\cJ(\wedge^k{\mn})\otimes\mV$ corresponds to the tensor product, i.e.
\begin{eqnarray*}
\left(A\beta\right)(U)&=&-(-1)^{|A||U|}\beta(AU)+(-1)^{|A||U|}A^\mV\beta(U)\qquad \mbox{for }A\in\mg,\quad \beta\in\cJ(\wedge^k{\mn})\otimes\mV\mbox{ and }U\in\cU(\mg).
\end{eqnarray*}

We define the $\mg$-module morphism $d^\cJ:\cJ(\wedge^k{\mn})\otimes\mV\to \cJ(\wedge^{k+1}{\mn})\otimes\mV$ to be the operator $d:\cJ(\wedge^k{\mn})\to \cJ(\wedge^{k+1}{\mn})$ as defined in Theorem \ref{thmdefop} in case $\mV=0$ which is extended trivially to $\cJ(\wedge^k{\mn})\otimes\mV$.

\begin{theorem}
\label{twisted}
The $\mg$-module morphism $\chi$ between $\cJ(C^k(\overline{\mn},\mV))$ and $\cJ(\wedge^k{\mn})\otimes\mV$ which sends $\alpha\in\cJ(C^k(\overline{\mn},\mV))$ to $\widetilde{\alpha}\in\cJ(\wedge^k{\mn})\otimes\mV$ defined as
\begin{eqnarray*}
\widetilde{\alpha}(U)=\sum_j \left(U^j_1\right)^\mV\left( \alpha(U^j_2)\right)
\end{eqnarray*}
is an isomorphism. Moreover it satisfies the property $d=\chi^{-1}\circ d^\cJ\circ\chi$ with $d^\cJ$ as defined above.
\end{theorem}
\begin{proof}
The calculation
\begin{eqnarray*}
(A\widetilde{\alpha})(U)&=&(-1)^{|A||U|}A^\mV \widetilde{\alpha}(U)-\sum_j(-1)^{|A||U|}A^\mV \left(U_1^j\right)^\mV{\alpha}(U^j_2)-\sum_j(-1)^{|A||U|+|A||U^j_1|} \left(U_1^j\right)^\mV{\alpha}(AU^j_2)\\
&=&-\sum_j(-1)^{|U^j_2||A|}\left(U^j_1\right)^{\mV}\alpha(A U^j_2)=\widetilde{A\alpha}(U)
\end{eqnarray*}
shows that the linear map $\chi$ is a $\mg$-module morphism. The calculation
\begin{eqnarray*}
\widetilde{\alpha}(UZ)&=&\sum_j(-1)^{|U_2^j||Z|}\left(U_1^j\right)^\mV Z^\mV\alpha(U_2^j)-\sum_j(-1)^{|U_2^j||Z|}\left(U_1^j\right)^\mV Z \left(\alpha(U_2^j)\right)\\
&=&-\sum_j(-1)^{|U||Z|}Z^\cJ\left(U_1^j\right)^\mV  \left(\alpha(U_2^j)\right)=-(-1)^{|U||Z|}Z^\cJ\left(\widetilde{\alpha}(U)\right)
\end{eqnarray*}
for $Z\in\Mp$ shows that the image of $\chi$ is inside $\cJ(\wedge^k{\mn})\otimes\mV$.

The inverse of $\chi$ is defined as $\chi^{-1}(\beta)(U)=\sum_j S(U^j_1)^\mV\beta(U^j_2)$ for $\beta\in\cJ(\wedge^k{\mn})\otimes\mV$ and $U\in\cU(\mg)$. The proof that this is the inverse follows immediately from the relation
\begin{eqnarray*}
(m\otimes 1)\circ(S\otimes \Delta)\circ\Delta=(m\otimes 1)\circ((S\otimes1)\circ \Delta\otimes 1)\circ\Delta=(\varepsilon\otimes 1)\circ\Delta=1.
\end{eqnarray*}

Also the fact that $d=\chi^{-1}\circ d^\cJ\circ\chi$ holds follows from a direct calculation and the relation
\begin{eqnarray*}
\partial f&=&\partial^\cJ f+(k+1)X_a\wedge (X_a^\dagger)^\mV f
\end{eqnarray*}
for $f\in C^k(\overline{\mn},\mV)$ and $\partial^\cJ$ the standard derivative on $C^k(\overline{\mn},0)=\Lambda^k\mn$ trivially extended to $C^k(\overline{\mn},\mV)$, which follows from the proof of Lemma \ref{propstandder}.
\end{proof}

The vector spaces $\cJ(\Lambda^k\mn)=\mbox{Hom}_{\cU(\Mp)}(\cU(\mg),\Lambda^k\mn)$ are naturally isomorphic to Hom$(\cU(\overline{\mn}),\Lambda^k\mn)$ by the Poincar\'e-Birkhoff-Witt theorem and the operator $d$ from Theorem \ref{thmdefop} for $\mV=0$ remains identically defined under this identification. The space Hom$(\cU(\overline{\mn}),\mC)$ becomes an algebra with multiplication defined by $(\alpha\beta)(U)=\sum_j\alpha(U^j_1)\beta(U^j_2)$ with $\alpha,\beta\in$Hom$(\cU(\overline{\mn}),\mC)$ and $U\in\cU(\overline{\mn})$, where now we consider the Hopf superalgebra $\cU(\overline{\mn})$. This multiplication extends trivially to the case where $\alpha\in$Hom$(\cU(\overline{\mn}),\mC)$ and $\beta\in$Hom$(\cU(\overline{\mn}),\Lambda^k\mn)$. Then the operator $d$ can be rewritten as in the following theorem.
\begin{theorem}
\label{extder}
For $d_k:$ {\rm Hom$(\cU(\overline{\mn}),\Lambda^k\mn)\to$ Hom$(\cU(\overline{\mn}),\Lambda^{k+1}\mn)$} induced from the operator $d$ in Theorem \ref{thmdefop}, there are elements $\theta_a\in$ {\rm Hom}$(\cU(\overline{\mn}),\mn)$ such that $\theta_a(1)=X_a$ and 
\begin{eqnarray*}
d_k \circ\left(\theta_a\wedge\right)&=&\left(\theta_a\wedge\right)\circ d_{k-1}
\end{eqnarray*}
holds and $d_0=\sum_a\theta_a\partial_{x_a}$ with $\partial_{x_a}$ supercommuting endomorphisms on {\rm Hom}$(\cU(\overline{\mn}),\mC)$ satisfying the Leibniz rule.
\end{theorem}
It is clear that for the example of $\mg=\mathfrak{gl}(m|n)$ and the case where $\mg_{\overline{0}}=\mathfrak{gl}(m)\oplus\mathfrak{gl}(n)$ is contained in the parabolic subalgebra $\Mp$, this property is immediate since then the radical $\mn$ is a supercommutative Lie superalgebra.
\begin{proof}
As a vector space the isomorphism Hom$(\cU(\overline{\mn}),\mC)\cong S(\mn)$ holds, with $S(\mn)=\oplus_{j=0}^\infty S^j(\mn)$ the supersymmetric tensor powers of $\mn$. It then follows that all endomorphisms on Hom$(\cU(\overline{\mn}),\mC)$ satisfying the Leibniz rule can be written in terms of a commuting basis $\{\partial_{x_b}\}$, in an expansion with Hom$(\cU(\overline{\mn}),\mC)$-valued coefficients. Therefore the operations $\partial_{X^\dagger_a}$ defined as
\begin{eqnarray*}
\left(\partial_{X^\dagger_a}\alpha\right)(U)&=&(-1)^{|U||X_a|}\alpha(UX^\dagger_a)
\end{eqnarray*}
can be expanded as $\partial_{X^\dagger_a}=\sum_b f_{ab}\partial_{x_b}$ for $f_{ab}\in$ Hom$(\cU(\overline{\mn}),\mC)$. Then the elements $\theta_b$ of Hom$(\cU(\overline{\mn}),\mn)$ are defined by $\theta_b=\sum_a X_a f_{ab}$. By the fact that $d_1\circ d_0=0$, see proof of Theorem \ref{Rhamcomplex}, it follows that $d_1(\theta_a)=0$ since $\theta_a=d_0(x_a)$ with $x_a\in$Hom$(\cU(\overline{\mn}),\mC)$ canonically defined by the operators $\{\partial_{x_b}\}$. Then it follows easily that $d_k(\theta_a\wedge\alpha)=d_1(\theta_a)\wedge\alpha-\theta_a\wedge d_{k-1}(\alpha)$, which concludes the proof.
\end{proof}



\end{document}